\documentclass{amsart}

\newtheorem{theorem}{Theorem}[section]
\newtheorem{lemma}[theorem]{Lemma}

\theoremstyle{definition}
\newtheorem{definition}[theorem]{Definition}

\theoremstyle{remark}
\newtheorem{remark}[theorem]{Remark}

\numberwithin{equation}{section}

\def\R{\hbox{\bf R}}

\def\N{\hbox{\bf N}}

\def\eps{\varepsilon}

\newcommand{\rf}[1]{(\ref{#1})}

\newcommand\DD{\mathcal{D}}
\newcommand\HH{\mathcal{H}}

\renewcommand{\N}{{\mathbb N}}
\renewcommand{\R}{{\mathbb R}}

\newenvironment{Proofc}[1]{\smallskip\par\noindent\textsc{#1}~}%
  {\hfill$\Box$\bigskip\par}

\newcommand{\Da}{\Lambda^\alpha}

\nonstopmode

\newcommand{\be}{\begin{equation}} 
\newcommand{\ee}{\end{equation}}
\newcommand{\bea}{\begin{eqnarray}} 
\newcommand{\eea}{\end{eqnarray}}
\newcommand{\dx}{\,{\rm d}x}
\newcommand{\dy}{\,{\rm d}y}
\newcommand{\dz}{\,{\rm d}z}
\newcommand{\dt}{\,{\rm d}t}
\newcommand{\la}{\langle}
\newcommand{\ra}{\rangle} 
\newcommand{\uni}{\rm 1\!\!I}

\begin{document}
 \baselineskip=16pt 

\title[Nonlinear diffusion of dislocation density]{Nonlinear diffusion of dislocation density\\
  and self-similar solutions}

\author{Piotr Biler}
\address{Instytut Matematyczny, Uniwersytet Wroc\l awski,
 pl. Grunwaldzki 2/4, 50-384 Wroc\l aw, Poland}
\email{Piotr.Biler{@}math.uni.wroc.pl}

\author{Grzegorz Karch}
\address{Instytut Matematyczny, Uniwersytet Wroc\l awski,
 pl. Grunwaldzki 2/4, 50-384 Wroc\l aw, Poland}
\email{Grzegorz.Karch{@}math.uni.wroc.pl}
\urladdr{http://www.math.uni.wroc.pl/$\sim$karch}

\author{R\'egis Monneau}
\address{CERMICS, \'Ecole nationale des Ponts et Chauss\'ees, 
6 et 8 avenue Blaise Pascal, Cit\'e Descartes, Champs-sur-Marne,
 77455 Marne-la-Vall\'ee Cedex 2, France}
\email{monneau@cermics.enpc.fr} 
\urladdr{http://cermics.enpc.fr/$\sim$monneau}

\date{\today}

\thanks{The third author is indebted to Cyril Imbert for stimulating and enlightening discussions on the subject of this paper. 
This work was supported by the contract ANR MICA (2006-2009), 
by the European Commission Marie Curie Host Fellowship for the Transfer of Knowledge 
``Harmonic Analysis, Nonlinear Analysis and Probability''  MTKD-CT-2004-013389, 
and by the Polish Ministry of Science grant N201 022 32/0902. 
}

\begin{abstract} 
We study a nonlinear pseudodifferential equation describing the dynamics of dislocations. 
The long time asymptotics of solutions is described by the self-similar profiles. 
\end{abstract}

\keywords{One-dimensional nonlinear pseudodifferential equation, viscosity solutions, self-similar solutions, stability} 

\subjclass[2000]{35K, 35S, 35B40, 70}


\maketitle

\section{Introduction}

In this paper, we study the following initial value problem for the 
nonlinear and nonlocal equation
\begin{eqnarray}\label{eq}
&&u_t=-|u_x| \  \Da u \quad \mbox{on}\quad \R\times (0,+\infty),\\
&&u(x,0)=u_0(x) \quad \mbox{for}\quad x\in \R,\label{ini}
\end{eqnarray}
where the assumptions on the initial datum $u_0$ will be precised later. Here, for $\alpha\in (0,2)$,
$
\Da=\left({\partial^2}/{\partial x^2}\right)^{\alpha/2}
$
is the pseudodifferential operator defined  {\it via} the Fourier
transform
\begin{equation}\label{Da}
\widehat {(\Da w)}(\xi)= |\xi|^\alpha \widehat w(\xi).
\end{equation}

For $\alpha=1$, equation \rf{eq} is an integrated form of a model studied by Head 
\cite{H} for the self-dynamics of a
dislocation density represented by $u_x$. This model is a~mean field model
that has been derived rigorously in \cite{FIM} as the limit of a system of particles in
interactions with forces in $\frac{1}{x}$.
In this model, dislocations can be of two types, $+$ or $-$, depending on the sign of
their Burgers vector (see the book by Hirth and Lothe \cite{HLo} for a physical definition of the Burgers vector). Here, the density $u_x$ means the positive
density $|u_x|$ of dislocations of type of the sign of $u_x$.
Moreover, the occurrence of the absolute value $|u_x|$ in the equation allows the vanishing of
dislocation particles of opposite sign. In the present paper, we study the general case $\alpha \in (0,2)$ that
could be seen as a mean field model of particles with repulsive interactions in
$\frac{1}{x^\alpha}$.

Recall that the operator $\Da$  has the
L\'evy-Khintchine  integral representation for every $\alpha\in (0,2)$
\begin{equation}\label{rep}
-\Da w(x)=C(\alpha)\int_{\R}\left\{w(x+z)-w(x)-zw'(x)\uni_{\left\{|z|\le 1\right\}}\right\}
\frac{\dz}{|z|^{1+\alpha}},
\end{equation}
where $C(\alpha)>0$ is a constant.
This formula (discussed in, {e.g.}, \cite[Th. 1]{DI} for functions $w$ in the
Schwartz space) allows us to extend the definition of $\Da$ to functions which are bounded
and sufficiently smooth, however, not necessarily decaying at infinity.

\subsection{Main results}

First note that equation \rf{eq} is invariant under the scaling
\begin{equation}\label{resc} 
u^\lambda(x,t)=u(\lambda x,\lambda^{\alpha+1}t)
\end{equation} 
for each $\lambda>0$ which means  that if $u=u(x,t)$ is a solution to \rf{eq}, 
then $u^\lambda=u^\lambda(x,t)$ is so.
Hence, our first goal is to construct  self-similar solutions of equation 
\rf{eq}, {i.e.} solutions which are invariant under the scaling 
\rf{resc}. By a standard argument, any self-similar solution should have 
the following form    
\begin{equation}\label{eq::ss}
{u_\alpha(x,t)=\Phi_\alpha(y)\quad \mbox{with}\quad 
 y=\frac{x}{t^{1/(\alpha+1)}}},  
\end{equation} 
where the self-similar profile  $\Phi_\alpha$ 
has to satisfy the following equation
\begin{equation}\label{eq::sse}
-(\alpha+1)^{-1}\  y \ \Phi_\alpha'(y) = - (\Da \Phi_\alpha(y))
\ \Phi_\alpha'(y)   \quad \mbox{for all}\quad y\in \R.  
\end{equation}

In our first theorem, we construct solutions to equation \rf{eq::sse}.

\begin{theorem}[Existence of  self-similar profile]\label{th:m:1}
Let $\alpha\in (0,2)$. There exists a~nondecreasing function $\Phi_\alpha$ 
of  the regularity $C^{1+\alpha/2}$ at each point and  analytic on 
$(-y_\alpha, y_\alpha)$    
 for some $y_\alpha >0$, which satisfies
$$\Phi_\alpha = \left\{\begin{array}{l}
0 \quad \mbox{\rm on}\quad (-\infty ,-y_\alpha),\\
1 \quad \mbox{\rm on}\quad (y_\alpha,+\infty),
\end{array}\right.$$
and
$$
(\Da \Phi_\alpha)(y) = \frac{y}{\alpha +1} \quad\mbox{\rm for all} \quad
y\in (-y_\alpha,y_\alpha).
$$
\end{theorem}

\begin{remark}
We can obtain the self-similar solutions corresponding to
different boundary values at infinity, simply considering for any $\gamma>0$ 
and $b\in \R$ the profiles
$
{\gamma\Phi_\alpha\left(\gamma^{-1/(\alpha+1)}y\right)+b}
$
which are also solutions of equation \rf{eq::sse}.
\end{remark}

\begin{remark}
The fact that $\partial_y \Phi_\alpha$ has compact support reveals a  finite velocity propagation 
of the support of the solution which is typical for
 solutions the porous 
medium equation, cf. Remark \ref{rem:porous}, below.
\end{remark}

At least formally, the function $\Phi_\alpha$ is the solution of
(\ref{eq::sse}), and the self-similar function $u_\alpha$ given by (\ref{eq::ss}) is a solution of equation \rf{eq}
with the initial datum being the Heaviside function  
\begin{equation}\label{eq::id}
u_0(x)=H(x)=\left\{\begin{array}{ll}
0&\quad \mbox{\rm if}\quad x<0,\\
1&\quad \mbox{\rm if}\quad x>0.
\end{array}\right.   
\end{equation}
In order to check that $u_\alpha$ given by \rf{eq::ss} solves \rf{eq},
we introduce a suitable notion of viscosity solutions  to the initial value problem \rf{eq}--\rf{ini},  see Section \ref{sectvs}. In this setting, we show
in Theorem \ref{th:exis} the existence and the uniqueness of a
solution for any initial condition  $u_0$ in $BUC(\R)$, 
{i.e.}~the space of bounded and uniformly continuous functions on $\R$.
Although, the initial datum \rf{eq::id} is not continuous, 
we have the following result.

\begin{theorem}[Uniqueness of the self-similar solution]\label{th:m:2}
Let $\alpha\in (0,2)$. Then the function $u_\alpha$ defined in \rf{eq::ss}
with the profile 
$\Phi_\alpha$ constructed in Theorem
\ref{th:m:1} 
is the unique viscosity solution of equation \rf{eq} with the 
initial datum \rf{eq::id}.
\end{theorem}

In Theorem \ref{th:m:2}, the uniqueness holds in the sense that if $u$ 
is another viscosity solution to \rf{eq}, \rf{eq::id}, then
$u=u_\alpha$
 on $(\R\times [0,+\infty)) \setminus \{(0,0)\}.$

The self-similar solutions are not only unique, but are also stable in this
framework of viscosity solutions, as the following result shows.

\begin{theorem}[Stability of the self-similar solution]\label{th:m:3}
Let $\alpha\in (0,2)$. For any initial data $u_0\in BUC(\R)$ satisfying
\begin{equation}\label{eq::101}
\lim_{x\to -\infty}u_0(x)=0\quad \mbox{\rm and}\quad \lim_{x\to +\infty}u_0(x)=1,   
\end{equation}
let us consider the unique viscosity solution $u=u(x,t)$ of \rf{eq}--\rf{ini}
 and, for each $\lambda>1$, its rescaled version $u^\lambda =u^\lambda(x,t)$ 
given by equation \rf{resc}.
Then, for any compact set 
$K\subset \left(\R\times [0,+\infty)\right) \setminus
\left\{(0,0)\right\}$, 
we have
\begin{equation}\label{conv:self}
{u^\lambda(x,t) \to \Phi_\alpha\left(\frac{x}{t^{1/(\alpha
        +1)}}\right)} \quad \mbox{\rm in}\quad L^\infty(K) 
\quad \mbox{\rm as}\quad
        \lambda \to +\infty.
\end{equation}
\end{theorem}

Theorem \ref{th:m:3} contains a result on  the long time behaviour of 
 solution because, first, choosing $t=1$ in \rf{conv:self} and, next, substituting 
$\lambda =t^{1/(\alpha+1)}$ we obtain  
 the convergence of $u\left(xt^{1/(\alpha+1)}, t\right)$ toward the self-similar profile $\Phi_\alpha(x)$.  

On the other hand,  convergence  \rf{conv:self} can be seen 
as a stability result when we consider initial data which are perturbations
of the Heaviside function. This is a~nonstandard stability result in
the framework of discontinuous viscosity solutions. It shows that the
 approach by viscosity solutions  is a good one in the sense of
Hadamard, even if we consider here initial conditions which 
are perturbations of
the Heaviside function.

Finally, we have the following result of independent interest.

\begin{theorem}[Optimal decay estimates]\label{th:m:4}
Let $\alpha\in (0,1]$. For any initial condition $u_0\in BUC(\R)$ such that $u_{0,x}\in
L^1(\R)$, the unique viscosity solution $u$ of \rf{eq}--\rf{ini}  satisfies
$$
\|u(\cdot,t)\|_\infty \le \|u_{0}\|_\infty\quad\mbox{and}\quad
\|u_x(\cdot,t)\|_\infty \le \|u_{0,x}\|_\infty \quad \mbox{\rm for any}\quad t> 0.$$
Moreover, for every $p\in [1,+\infty)$ we have
\begin{equation}\label{eq::decay}
\|u_x(\cdot,t)\|_p \le C_{p,\alpha}\|u_{0,x}\|_1^{\frac{p\alpha +1}{p(\alpha+1)}} \ t^{-\frac{(p-1)}{p(\alpha+1)}}
\quad \mbox{\rm for  any}\quad t> 0,   
\end{equation}
with some constant $C_{p,\alpha}>0$ depending only on $p$ and $\alpha$.
\end{theorem}

The decay given in (\ref{eq::decay}) is optimal in the sense that the self-similar
solution satisfies
$\|(u_\alpha)_x(\cdot ,t)\|_p  =  \|(\Phi_\alpha)_y(\cdot)\|_p\
t^{-\frac{(p-1)}{p(\alpha+1)}}.$

\begin{remark}\label{rem:porous}
The equation satisfied by $v=u_x$ of the following form 
\begin{equation}
v_t=(|v|\Lambda^{\alpha-1}\mathcal{H}v)_x \label{eq:v:rem}
\end{equation}
(with the Hilbert transform denoted by $\mathcal{H}$) can be treated as the 
nonlocal counterpart of the porous medium equation. Indeed, for $\alpha=2$ and
for nonnegative $v$, equation \rf{eq:v:rem} reduces to
$
v_t=(vv_x)_x =\left({v^2}/{2}\right)_{xx}.
$
As in the case of the porous medium equation (see {\it e.g.} \cite{JLV}
and the references therein),
 estimates (\ref{eq::decay})  show a regularizing effect
created by the equation, even for the anomalous diffusion:  
if $ v_{0}\in L^1(\R)$ then $ v\in L^p(\R)$ for each $p>1$. 
Observe also that equation \rf{eq:v:rem} has the compactly supported
self-similar solution 
$v(x,t)=t^{-\frac{1}{\alpha+1}}\Phi'_\alpha\left(x/t^{\frac{1}{\alpha+1}}\right)$, 
where the profile $\Phi_\alpha$ was constructed in Theorem \ref{th:m:1}.
This function for $\alpha=2$ corresponds to the well-known Barenblatt-Prattle 
solution of the porous medium equation. 
\end{remark}

\begin{remark}\label{product}
For $\alpha \in (1,2)$, we do not know how to define the product
$|u_x|\ (\Da u)$ in the sense of distributions, which is an obstacle for us
to prove the result of Theorem~\ref{th:m:4}  in this case, see Section \ref{s6}.
Note, however, that the inequalities from Theorem \ref{th:m:4} are valid for 
$\alpha\in (1,2]$ as well, provided the solution $u=u(x,t)$ is sufficiently 
regular.
\end{remark}

\subsection{Organization of the paper}

In Section \ref{s2}, we construct explicitly the self-similar solution. In
Section \ref{sectvs}, we recall the necessary material about viscosity
solutions, which will be used in the remainder of the paper. In Section
\ref{s4}, we prove the uniqueness of the self-similar solution. Under the
additional assumption that the solution is confined between its boundary
values at infinity, we prove the stability of the self-similar solution, namely
Theorem \ref{th:m:3}.
In Section \ref{s5}, we prove further decay properties of a
solution with compact support. Applying these estimates, we finish the proof of Theorem
\ref{th:m:3} in the general case. In Section \ref{s6}, we introduce an
$\varepsilon$-regularized equation, for which we prove both 
the global existence of a smooth solution
and the corresponding gradient estimates.
Finally in
Section \ref{s7}, we deduce the gradient estimate in the
limit case $\varepsilon=0$, namely Theorem \ref{th:m:4}, using the corresponding
estimates for the approximate $\varepsilon$-problem.

\section{Construction of self-similar solutions}\label{s2}

\begin{proof}[Proof of Theorem \ref{th:m:1}]
The crucial role in the construction of the self-similar profile $\Phi_\alpha$ 
is played by the function
\begin{equation}\label{func:v}
v(x)=\left\{
\begin{array}{ccc}
K(\alpha)\left(1-|x|^2\right)^{\alpha/2}& \mbox{for}&|x|< 1,\\
0& \mbox{for}&|x|\ge 1,
\end{array}\right.
\end{equation}
with $K(\alpha)=\Gamma(1/2)\left[2^\alpha\Gamma(1+{\alpha}/{2})
\Gamma((1+\alpha)/{2})\right]^{-1}$.
This function 
(together with its multidimensional counterparts)
has an important 
probabilistic interpretation. 
Indeed, if $\{X(t)\}_{t\geq 0}$ denotes the symmetric $\alpha$-stable 
process in $\R$ 
of order  $\alpha\in (0,2]$ and if  $ T=\inf\{t\,:\,|X(t)|>1\}$ 
is the first passage time of the process to the exterior of the
ball $\{x:\ |x|\leq 1\}$, Getoor \cite{G}
 proved that $\mathbb{E}^x(T)=v(x)$,
where $\mathbb{E}^x$ denotes the expectation under the condition $X(0)=x$.

In particular, it was computed in \cite[Th.~5.2]{G}  using a purely 
analytical argument (based on definition \rf{Da} and on properties of 
the Fourier transform) that
 $\Da v\in L^1(\R)$ and 
\begin{equation}\label{eq:Getoor}
\Da v(x)=1\quad \mbox{for}\quad |x|<1.
\end{equation}

Now, for the function $v$, we  define the bounded, nondecreasing, $C^{1+\alpha/2}$-function
$$u(x)=\int_0^x  v(y) \dy$$
which obviously satisfies
$u(x)=M(\alpha)$ for all $x\geq 1$ and $u(x)=-M(\alpha)$ for $x\leq -1$ with 
$$
M(\alpha)=K(\alpha)\int_0^1\left(1-|y|^2\right)^{\alpha/2} \dy
=
\frac{\pi}{2^\alpha(\alpha+1)\Gamma\left(\frac{1+\alpha}2\right)^2}.
$$

Then, for any $\varphi\in C^\infty_c(\R)$,  we can
introduce  the following duality 
$$\la \Da u,\varphi\ra = \int_{\R} u(y) 
(\Da \varphi)(y)\dy.$$
This defines $\Da u$ as a distribution, because we can check (using the L\'evy-Khintchine formula
(\ref{rep})) that there exists a constant $C>0$ such that
$$\displaystyle{|(\Da \varphi)(x)|\le \frac{C
    \|\varphi\|_{W^{2,\infty}(\R)}}{1+|x|^{1+\alpha}}}.$$
If, moreover,  ${\rm supp}\, \varphi\subset (-1,1)$, 
it is easy to check using the properties of the function $v=v(x)$
that
$$\la \partial_x(\Da u),\varphi\ra
=-\la u, \Da (\partial_x \varphi)\ra= -\la u, \partial_x(\Da \varphi)\ra
=\la \Da (\partial_x u),\varphi\ra=\la 1,\varphi\ra,$$
where the last inequality is a consequence of \rf{eq:Getoor}.
From the symmetry of $v$, we deduce the antisymmetry of $u$, and then 
$(\Da u)(-x)=-(\Da u)(x)$. Therefore, we get the equality
$(\Da u)(x)=x$ in
$ {\DD}'(-1,1)$, however by \cite[Cor. 3.1.5]{Hor}, in the classical sense for each $y\in (-1,1)$, too.

Finally, we define the nonnegative function
$$
\Phi_\alpha(y)=\frac{\gamma}{\alpha+1} \left\{u\left(\gamma^{-1/(\alpha
        +1)}y\right)+M(\alpha)\right\} \quad \mbox{with}\quad 
\gamma^{-1} =\frac{2M(\alpha)}{\alpha+1}.
$$
Now, for $y_\alpha= \gamma^{1/(\alpha+1)}=[2M(\alpha)]^{-1/(\alpha+1)}$, 
we can check easily that $\Phi_\alpha$ is exactly as stated in 
Theorem \ref{th:m:1}, which ends the proof. 
\end{proof}
\bigskip

Let us note that we will not use in the sequel the explicit form of the function
$\Phi_\alpha$, but only its properties listed in Theorem \ref{th:m:1}.

\begin{remark}\label{Head}
{\rm
It is known since the work of Head and Louat \cite{HL} (see also \cite{H}) 
that the function $v(x)=K\left(1-x^2\right)^{1/2}$ (with a suitably chosen 
constant $K=K(1)>0$) is the solution of the
equation $(\Lambda^1 v)(x)=1$ on $(-1,1)$.
This result is a consequence of an inversion theorem due to Muskhelishvili, 
see either \cite[p. 251]{M1} or \cite[Sec. 4.3]{T}.  }
\end{remark}

\section{Notion of viscosity solutions}\label{sectvs}

Here, we consider equation \rf{eq} and its vanishing
viscosity approximation, {i.e.} the following initial value problem 
for $\alpha\in(0,2)$ and $\eta\geq 0$
\begin{eqnarray}
&&u_t=\eta  u_{xx} -|u_x| \ \Da u \quad \mbox{on}\quad \R\times (0,+\infty),\label{eq:eta}\\
&&u(x,0)=u_0(x) \quad \mbox{for}\quad x\in \R. \label{ini:eta}
\end{eqnarray}

In this section, we present the framework of viscosity solutions to 
problem \rf{eq:eta}--\rf{ini:eta}. To
this end, we recall briefly the necessary material, which can be either
found in the literature or is essentially a standard adaptation of
those results. We also refer the reader to Crandall {\it et al.} \cite{CIL} for a classical text on viscosity solutions to local ({i.e.} partial differential) equations.

Let us first recall the definition of  relaxed {\em lower
semi-continuous} (lsc, for short) and {\em upper semi-continuous} (usc, for
short) {\em limits} of a family of functions $u^\eps$ which is locally
bounded uniformly with respect to $\eps$ 
\begin{equation*}
\limsup_{\varepsilon\to 0}{}^* u^\eps (x,t) = 
\limsup_{\substack{\eps \to 0 \\  y \to x,  s \to t}}  u^\eps (y,s)   \;\;\; \mbox{and} \;\;\;  
\liminf_{\varepsilon\to 0}{}_* u^\eps (x,t) = 
\liminf_{\substack{\eps \to 0 \\  y \to x, s \to t}}  u^\eps (y,s).
\end{equation*}
If the family consists of a single element, we recognize the usc envelope
and the lsc envelope of a locally bounded function $u$ 
$$
u^* (x,t) = \limsup_{y \to x, s \to t}  u (y,s)  \quad \mbox{ and } \quad   u_* (x,t) =
\liminf_{y \to x,s \to t}  u (y,s). 
$$

Now, we recall the definition of a {\em viscosity solution} for \rf{eq:eta}--\rf{ini:eta}.
Here, the difficulty is caused by the measure $|z|^{-1-\alpha}\dz$ appearing in the
L\'evy-Khintchine formula (\ref{rep}) which  is singular at the origin and, consequently,
the function has to be at least $C^{1,1}$ in space in order that $\Da
u(\cdot,t)$ makes sense (especially for $\alpha$ close to $2$). 
We refer the reader, for instance, to \cite{sayah91a, BI, JK2} for the
stationary case, and to \cite{JK1, IMR} for the evolution equation where this 
question is discussed in detail.

Now, we are in a position to define viscosity solutions.

\begin{definition}[Viscosity solution/subsolution/supersolution]
A bounded usc  (resp.~lsc) function  $u : 
\R \times \R^+ \to \R$ is a \emph{viscosity subsolution} 
(resp.~\emph{supersolution}) of equation \eqref{eq:eta} 
on $\R\times (0,+\infty)$ if for any point
$(x_0,t_0)$ with  $t_0>0$, any $\tau \in (0,t_0)$, and any test function $\phi$ belonging to
$C^2(\R\times (0,+\infty)) \cap L^{\infty}(\R\times (0,+\infty))$ such
that $u-\phi$ attains a maximum (resp. minimum) at the point $(x_0,t_0)$ on the
cylinder  
$$Q_{\tau}(x_0,t_0):= \R \times  (t_0-\tau,t_0+\tau),$$ 
 we have
$$
\partial_t \phi (x_0,t_0) -\eta  \phi_{xx} (x_0,t_0) +  
|\phi_x
(x_0,t_0)|\; (\Da \phi (\cdot,t_0))(x_0)  \le 0
\quad \mbox{\rm (resp. $\ge 0$)},
$$
where $(\Da \phi (\cdot,t_0))(x_0)$ is given by the L\'evy-Khintchine
formula (\ref{rep}).

We say that $u$ is a \emph{viscosity subsolution} 
(resp. \emph{supersolution}) of problem \eqref{eq:eta}--\rf{ini:eta}
 on $\R\times [0,+\infty)$,
if it satisfies moreover at time $t=0$ 
$$
u(\cdot,0) \le u_0^* \quad \left(\mbox{\rm resp.}\quad u(\cdot,0) \ge (u_0)_*\right).
$$

A function $u: \R \times \R^+ \to \R$ is a \emph{viscosity solution} of 
\eqref{eq:eta} on
$\R\times (0,+\infty)$ (resp. $\R\times [0,+\infty)$) if
 $u^*$ is a viscosity subsolution and $u_*$ is a viscosity
supersolution of the equation on $\R\times (0,+\infty)$ 
(resp. $\R\times [0,+\infty)$).
\end{definition}

Other equivalent definitions are also natural, see for instance \cite{BI}. 

\begin{remark}\label{rem::1}
Any bounded function $u\in C^{1+\beta}$ (with some $\beta>\max\{0,\alpha-1\}$)
which satisfies pointwisely (using the L\'evy-Khintchine formula \rf{rep}) equation (\ref{eq:eta}) with $\eta=0$, is indeed a~viscosity solution.   
\end{remark}

\begin{theorem}[Comparison principle]\label{th::1}
Consider a bounded usc subsolution $u$ and a bounded lsc
supersolution $v$ of \rf{eq:eta}--\rf{ini:eta}. 
If $u(x,0) \le u_0(x) \le v(x,0)$
for some $u_0 \in BUC(\R)$, then $u \le v$ on $\R \times [0,+\infty )$.  
\end{theorem}

\begin{proof}
Recall that in \cite[Th. 5]{IMR}, the comparison principle is proved
for $\alpha=1$ and $\eta =0$ under the additional assumption that $u_0\in  W^{1,\infty} (\R)$. Looking at
the proof of that result, the regularity of the initial data $u_0$ is only used
to show that 
\begin{equation}\label{eq::7}
\sup_{x\in\R} \left((u_0)^\varepsilon(x)-(u_0)_\varepsilon(x)\right) \to 0 \quad
\mbox{as}\quad \varepsilon \to 0,   
\end{equation}
where $(u_0)^\varepsilon$ and $(u_0)_\varepsilon$ are respectively sup and
inf-convolutions. It is easy (and classical) to check that (\ref{eq::7}) is
still true for $u_0 \in BUC(\R)$.
The general case can be done either considering a variation of the  proof
of \cite{IMR} taking into account the additional Laplace operator, or applying the ``maximum principle'' from
\cite{JK2}, or following, for instance, the lines of \cite{BI}. We skip here
the detail of this adaptation. 
This  finishes  the proof. 
\end{proof}


\begin{theorem}[Stability]\label{th:stab}
Let $\{u^\varepsilon\}_{\varepsilon>0}$ be a sequence of viscosity subsolutions
(resp. supersolutions) of equation (\ref{eq:eta})  which are locally bounded,
uniformly in $\varepsilon$. Then $\overline{u}=\limsup^* u^\varepsilon$
(resp. $\underline{u}=\liminf_* u^\varepsilon$) is a~subsolution
(resp. supersolution) of (\ref{eq:eta}) on $\R\times (0,+\infty)$.
\end{theorem}

\begin{proof}
A counterpart  of Theorem \ref{th:stab} is proved in \cite[Th.1]{BI}.
Here, the result for the time dependent problem is again a classical 
adaptation of that argument, so we skip  details.  
\end{proof}

\begin{remark}\label{remark:plus}
One can  generalize directly Theorem \ref{th:stab} assuming that 
$\{u^\varepsilon\}_{\varepsilon>0}$ are solutions to the sequence of equations
\rf{eq:eta} with $\eta=\varepsilon$. Then, in the limit 
$\varepsilon \to 0^+$, we obtain viscosity subsolutions
(resp. supersolutions) of equation (\ref{eq}). We use this property in 
the proof of Theorem \ref{th:m:4}.
\end{remark}

\begin{remark}\label{initialconditions}
In Theorem \ref{th:stab}, we only claim that the limit $\overline{u}$ is a
supersolution on $\R\times (0,+\infty)$, but not on $\R\times
[0,+\infty)$. In other words, we do not claim that $\overline{u}$ satisfies
the initial condition. Without further properties of the initial data $u^0$, it may happen that 
$\overline{u}(\cdot,0) \le u_0^*$ 
is not true.
\end{remark}


\begin{theorem}[Existence]\label{th:exis}
Consider $u_0 \in BUC(\R)$. Then there exists the unique bounded continuous viscosity solution $u$ of \rf{eq:eta}--\rf{ini:eta}. 
\end{theorem}

\begin{proof}
Applying the argument of \cite{I} (already adapted from the classical
arguments), we can construct a solution by the Perron method, if we are able to construct suitable barriers. 
\medskip

 {\it Case 1}: First, assume that  $u_0 \in W^{2,\infty} (\R)$.
Then the following functions
\begin{equation}\label{eq::200}
u_\pm(x,t)=u_0(x) \pm Ct   
\end{equation}
are barriers for $C>0$ large enough (depending on the norm
$\|u_0\|_{W^{2,\infty}(\R)}$),
and we get the existence of solutions by the Perron method.
\medskip

 {\it Case 2}: Let $u_0 \in BUC(\R)$.
For any $\varepsilon>0$, we can regularize $u_0$ by a convolution, and
get a function $u_0^\varepsilon \in W^{2,\infty}(\R)$ which satisfies, moreover,
\begin{equation}\label{eq::8}
|u_0^\varepsilon -u_0|\le \varepsilon.  
\end{equation}
Let us call $u^\varepsilon$ the solution of \rf{eq:eta}--\rf{ini:eta} with
the initial condition $u^\varepsilon_0$ instead of $u_0$. 
Then, from the fact
that the equation does not see the constants and from the
comparison principle (Theorem \ref{th::1}), we have for any
$\varepsilon,\,\delta >0$
$$|u^\varepsilon-u^\delta|\le \varepsilon +\delta.$$
Therefore, $\{u^\varepsilon\}_{\varepsilon>0}$ is the Cauchy sequence which
converges in $L^\infty(\R\times [0,+\infty))$ to some continuous function
$u$ (because all the functions $u^\varepsilon$ are continuous). By the
stability result (Theorem \ref{th:stab}), we see that $u$ is a viscosity
solution of equation \rf{eq:eta} on $\R\times (0,+\infty)$. To recover the initial boundary condition, we simply remark that
$u^\varepsilon(x,0)=u_0^\varepsilon(x)$ satisfies (\ref{eq::8}), and then
passing to the limit, we get
$u(x,0)=u_0(x).$
This shows that $u$ is a viscosity
solution of problem \rf{eq:eta}--\rf{ini:eta}
 on $\R\times [0,+\infty)$, and ends the proof of
 Theorem \ref{th:exis}. 
\end{proof}

\section{Uniqueness and stability of the self-similar solution}\label{s4}

\begin{lemma}[Comparison with the self-similar solution]\label{lem::104}
Let $v$ be a subsolution (resp.~a supersolution) of equation \rf{eq} with the Heaviside initial
datum given in (\ref{eq::id}). Then we have
$v^*\le  (u_\alpha)^*$ 
(resp. $(u_\alpha)_*\le  v_*$).
\end{lemma}

\begin{proof}
Using Remark \ref{rem::1} and properties of $\Phi_\alpha$ 
gathered in  Theorem \ref{th:m:1}, 
it is straightforward to check that the self-similar solution
$u_\alpha(x,t)$ given in (\ref{eq::ss}) is a viscosity solution of equation
\rf{eq:eta}--\rf{ini:eta} with the initial condition (\ref{eq::id}).

Now, we show the inequality $(u_\alpha)_*\le  v_*$.
Let  $v$ be a viscosity supersolution of \rf{eq:eta}--\rf{ini:eta} 
with the Heaviside initial datum  (\ref{eq::id}).
Given $a>0$ and $v^a(x,t)=v(a+x,t)$, we have
$$(u_\alpha)^*(x,0)\le (u_0)^*(x)\le (u_0)_*(a+x)\le v^a(x,0).$$
Because of the translation invariance 
of the equation \rf{eq}, we see that $v^a$ is still a
supersolution. Moreover, for any $a>0$,  we can always find an initial condition $u_a\in
BUC(\R)$ such that
$$u_\alpha(x,0)\le u_a(x)\le v^a(x,0).$$
Therefore, applying the comparison principle (Theorem \ref{th::1}), we
deduce that
$$u_\alpha\le v^a.$$
Because this is true for any $a>0$, we can take the limit as $a\to 0$ and
get $(u_\alpha)_* \le v_*$.

For a subsolution $v$, we proceed similarly to obtain 
$v^*\le  (u_\alpha)^*$. This finishes the proof of  Lemma \ref{lem::104}.
\end{proof}
\bigskip

\begin{proof}[Proof of Theorem \ref{th:m:2}]
We consider a viscosity solution $v$ of equation \rf{eq} with the Heaviside initial datum  \rf{eq::id}.
Using the both inequalities of Lemma \ref{lem::104}, and the fact that
$(u_\alpha)_*=(u_\alpha)^*$ on $(\R\times [0,+\infty)) \backslash
\left\{(0,0)\right\}$, we deduce the equality
$v=u_\alpha$ 
on $(\R\times [0,+\infty)) \setminus
\left\{(0,0)\right\},$
which ends the proof of  Theorem \ref{th:m:2}. 
\end{proof}
\medskip

We will now prove the following weaker version of Theorem \ref{th:m:3}.

\begin{theorem}[Convergence for suitable initial data]\label{th::110}
The convergence \rf{conv:self} in 
Theorem~\ref{th:m:3} holds true under the following additional assumption 
\begin{equation}\label{eq::100}
\lim_{y\to -\infty} u_0(y)=0   \le  u_0(x)     \le  1 = \lim_{y\to +\infty} u_0(y).
\end{equation}   
\end{theorem}

\begin{proof}
{\it Step 1: Limits after rescaling  of the solution.}
Consider a solution $u$ of \rf{eq}--\rf{ini}
 with an initial condition $u_0$
satisfying (\ref{eq::100}).
Recall that  for any $\lambda>0$,  the rescaled solution is 
given by $u^\lambda(x,t)=u(\lambda x,\lambda^{\alpha+1}t).$
 Let us define 
$$\overline{u}=\limsup_{\lambda\to +\infty }{}^* u^\lambda \quad \mbox{and}\quad
\underline{u}=\liminf_{\lambda\to +\infty }{}_* u^\lambda.$$
From the stability result (Theorem \ref{th:stab}), we know that $\overline{u}$
(resp. $\underline{u}$) is a subsolution (resp. supersolution) of
\rf{eq} on $\R\times (0,+\infty)$.
\medskip

 {\it Step 2: The initial condition.}
We now want to prove that 
\begin{equation}\label{eq::103}
\overline{u}(x,0)=\underline{u}(x,0)=H(x) \quad \mbox{for}\quad
x\in\R\backslash \left\{0\right\},   
\end{equation}
where $H$ is the Heaviside function.
To this end, we remark that $u_0$ satisfies for some $\gamma>0$ the inequality
$|u_0(x)|\le \gamma$ (note that $\gamma=1$ under assumption \rf{eq::100}), 
and for each $ \varepsilon >0$, there exists $M>0$ such that $|u_0(x)| < \varepsilon$ for $x\le -M$.

In particular, we get
$$u_0(x) < \varepsilon + \gamma H(x+M),$$
and then from the comparison principle, we deduce
\begin{equation}\label{eq::102}
u(x,t) \le \varepsilon + (u^\gamma_\alpha)_*\left(x+M,t\right)   
\end{equation}
with
\begin{equation}\label{eq::106}
u_\alpha^\gamma(x,t)=\Phi_\alpha^\gamma\left(\frac{x}{t^{1/(\alpha+1)}}\right)\quad
      \mbox{and}\quad \Phi_\alpha^\gamma(y)=\gamma
      \Phi_\alpha\left(\gamma^{-1/(\alpha+1)}y\right).   
\end{equation}
Here $\Phi_\alpha^\gamma$ is the self-similar profile solution of
      (\ref{eq::sse}) with the boundary conditions $0$ and $\gamma$ at
      infinity.
Moreover, because $u^\gamma_\alpha$ is continuous  off the origin, we
      can simply drop the star $*$\,, while we are interested in points
      different from the origin.
This implies
$$u^\lambda(x,t)\le \varepsilon +  \Phi^\gamma_\alpha\left(\frac{x+M
    \lambda^{-1}}{t^{1/(\alpha
      +1)}}\right),$$
and then
$$\overline{u}(x,t) \le \varepsilon +  \Phi^\gamma_\alpha\left(\frac{x}{t^{1/(\alpha
      +1)}}\right).$$
Therefore, for every $x<0$ we have
$$\overline{u}(x,0)\le \varepsilon + \Phi^\gamma_\alpha(-\infty)
=\varepsilon.$$
Because this is true for every $\varepsilon >0$, we get 
$\overline{u}(x,0)\le 0 \quad \mbox{for every}\quad x <0.$
We get the other inequalities similarly, and finally conclude that
(\ref{eq::103}) is valid.
\medskip 

 {\it 
Step 3: Initial condition at the origin, using assumption (\ref{eq::100}).}
We now make use of (\ref{eq::100}) to identify the initial values of the limits
$\overline{u}$ and $\underline{u}$. We deduce from the comparison principle 
that
$$0\le \underline{u}(x,0)\le \overline{u}(x,0) \le 1,$$
and then for every $x\in\R$ we have
$$\overline{u}(x,0) \le H^*(x) \quad \mbox{and}\quad \underline{u}(x,0) \ge
H_*(x).$$

 {\it Step 4: Identification of the limits after rescaling.}
From Lemma \ref{lem::104}, we obtain
$$\overline{u} \le (u_\alpha)^*=(u_\alpha)_* \le \underline{u} \quad \mbox{on}\quad (\R\times [0,+\infty)) \backslash
\left\{(0,0)\right\}.$$
We have by the construction $\underline{u}\le \overline{u}$,  hence we infer 
$$\overline{u} = \underline{u} = u_\alpha  \quad \mbox{on}\quad (\R\times [0,+\infty)) \backslash
\left\{(0,0)\right\}.$$

\medskip

 {\it Step 5: Conclusion for the convergence.}
Then for any compact $K\subset (\R\times [0,+\infty)) \backslash
\left\{(0,0)\right\}$, we can easily deduce that
$$\sup_{(x,t)\in K} |u^\lambda(x,t)-u_\alpha(x,t)| \to 0 \quad
\mbox{as}\quad \lambda \to +\infty,$$
which finishes the proof of  Theorem \ref{th::110}. 
\end{proof}

\section{Further decay properties and end of the proof of Theorem \ref{th:m:3}}\label{s5}

\begin{theorem}[Decay of a solution with  compact support]\label{pro::2}
Let $u$ be the solution to \rf{eq}--\rf{ini} with the initial datum $u_0\in
BUC(\R)$ satisfying for some $A>0$
\begin{equation}\label{eq::20}
u_0(x)\le 0 \quad \mbox{\rm for}\quad |x|\ge A.   
\end{equation}
Let also $\gamma>0$ be such that
$u_0(x) \le \gamma$   for all $ x\in \R.$
Then, there exist  $\beta, \beta' >0$ (depending on $\alpha$, but independent of $A,\gamma$) such that
$$u(x,t) \le C t^{-\beta},$$  
and
$$u(x,t) \le 0 \quad \mbox{\rm for}\quad |x|\ge C't^{\beta'}$$
with some constants $C=C(\alpha,A,\gamma)$ and $C'=C'(\alpha,A,\gamma)$.
\end{theorem}

First, we need the following 
\begin{lemma}[Decay after the first interaction]\label{lem:3}
Consider $\Phi_\alpha$ and $y_\alpha$  defined in Theorem \ref{th:m:1}.
Let $\nu \in (1/2,1)$ and $\xi_\nu \in (0,y_\alpha)$ be such that
$\Phi_\alpha(\xi_\nu)=\nu$. Let $T>0$ be defined by
\begin{equation}\label{eq::21}
\frac{A}{\gamma^{1/(\alpha+1)} T^{1/(\alpha
      +1)}} = \xi_\nu.   
\end{equation}
Then, under the assumptions of Theorem  \ref{pro::2}, we have
\begin{equation}\label{eq::22}
u(x,t) \le \nu \gamma \quad \mbox{\rm for all}\quad t\ge T, \quad x\in\R,   
\end{equation}
and
\begin{equation}\label{eq::23}
u(x,t) \le 0 \quad \mbox{\rm for all}\quad 0\le t\le T\quad \mbox{\rm and}\quad |x|\ge
A\left(1+\frac{y_\alpha}{\xi_\nu}\right). 
\end{equation}
\end{lemma}

\begin{proof}
Let us denote 
$\Phi_\alpha^\gamma(y)=\gamma\Phi_\alpha\left(\gamma^{-1/(\alpha+1)}y\right)$.
Then we have
$$\gamma H(x+A)  \ge u_0(x) \quad \mbox{for}\quad x\in \R,$$
where 
$$\gamma H(x+A) = \lim_{t\to 0^+}
\Phi_\alpha^\gamma\left(\frac{x+A}{t^{1/(\alpha +1)}}\right) \quad
\mbox{for}\quad x+A\not=0.$$

Now, we  apply the comparison principle to deduce that
$$\Phi_\alpha^\gamma\left(\frac{x+A}{t^{1/(\alpha +1)}}\right) \ge u(x,t)
\quad \mbox{for}\quad (x,t)\in \R\times (0,+\infty).$$
This argument can be made rigorous, simply, by replacing the function 
$\gamma H(x+A)$ by
$\Phi_\alpha^\gamma\left((x+A+\delta)/(t_\varepsilon^{1/(\alpha
      +1)})\right)$
for $\delta >0$ and some sequence $t_\varepsilon \to 0^+$, and then taking
the limit $\delta \to 0^+$. 

Therefore we have 
$$\gamma \Phi_\alpha\left(\frac{x+A}{\gamma^{1/(\alpha+1)} t^{1/(\alpha +1)}}\right) \ge u(x,t)
\quad \mbox{for}\quad (x,t)\in \R\times (0,+\infty).$$
From the properties of the support of $\Phi_\alpha$, we also deduce that
$$u(x,t) \le 0 \quad \mbox{for}\quad x\le -\left(A +
  y_\alpha (\gamma t)^{1/(\alpha+1)} \right),$$
and then, by symmetry,
$$u(x,t) \le 0 \quad \mbox{for}\quad |x|\ge A +
  y_\alpha (\gamma t)^{1/(\alpha+1)}. $$
Moreover, it follows from the monotonicity of $\Phi_\alpha$  that 
$$\gamma \Phi_\alpha\left(\frac{A}{(\gamma t)^{1/(\alpha+1)}}\right) \ge u(x,t)$$
for $x \le 0$, and by symmetry we can prove the same property for $x\ge 0$.
Then for $T>0$ defined in (\ref{eq::21}) we easily deduce (\ref{eq::22}) and (\ref{eq::23}).
This ends the proof of  Lemma \rf{lem:3}.
\end{proof}

\bigskip

\begin{proof}[Proof of Theorem \ref{pro::2}]
We apply recurrently Lemma \ref{lem:3}.
Define $A_0=A$, $\gamma_0=\gamma$,
and
$$A_{n+1}=A_n\left(1+\frac{y_\alpha}{\xi_\nu}\right),\quad \gamma_{n+1}=\nu
\gamma_n, \quad \mbox{and} \quad \frac{A_n}{(\gamma_n T_n)^{1/(\alpha+1)}} = \xi_\nu.$$
This gives
$$A_n=A_0\left(1+\frac{y_\alpha}{\xi_\nu}\right)^n,\quad \gamma_n
=\nu^n\gamma_0,\quad T_n=K\mu^n,$$
with
$$K=\frac{1}{\gamma_0}\left(\frac{A_0}{\xi_\nu}\right)^{\alpha+1},\quad
    1<\mu=\frac{1}{\nu}\left(1+\frac{y_\alpha}{\xi_\nu}\right)^{\alpha+1},$$
and therefore
$$u(x,t) \le \gamma_n \quad \mbox{for}\quad t\ge T_0 + ...+T_{n-1} =
K\frac{\mu^n-1}{\mu-1}.$$
In particular, we get for any $n\in\N$
$$u(x,t) \le \gamma_0 \nu^n \quad \mbox{for}\quad t\ge K_0 \mu^n$$
with $K_0=K/(\mu-1)$.
This implies
$$u(x,t)\le \gamma_0K_0^\beta t^{-\beta} \quad \mbox{for any} \quad t>0,
\quad x\in\R,$$
with 
$$\beta=-\frac{\ln \nu}{\ln \mu}>0.$$
Similarly, we have
$$u(x,t) \le 0 \quad \mbox{for}\quad |x|\ge A_{n} \quad \mbox{if}\quad t
\le T_0 + ...+T_{n-1}=K\frac{\mu^n-1}{\mu-1}.$$
In particular, we get for any $n\in\N\backslash \left\{0\right\}$
$$u(x,t) \le 0 \quad \mbox{for}\quad |x|\ge A_0
\left(1+\frac{y_\alpha}{\xi_\nu}\right)^{n}, \quad \mbox{if}\quad t\le
K_0'\mu^{n}$$
with $K_0'=K/\mu$. This implies
$$u(x,t) \le 0 \quad \mbox{for}\quad |x|\ge A_0 (K_0')^{-\beta'} t^{\beta'}
\quad \mbox{for}\quad t\ge 0,$$
with
$$\beta'=\frac{\ln \left(1+\frac{y_\alpha}{\xi_\nu}\right)}{\ln \mu}>0.$$
This ends the proof of Theorem \ref{pro::2}. 
\end{proof}

\bigskip

As a corollary, we can now remove assumption (\ref{eq::100}) in Theorem \ref{th::110} and
complete the proof of Theorem \ref{th:m:3}.

\begin{proof}[Proof of Theorem \ref{th:m:3}]
We simply repeat Step 3 of the proof of Theorem \ref{th::110}, but here
without assuming (\ref{eq::100}). Then, for any $\varepsilon >0$ there exists
$A>0$ such that
$$u_0(x)\le 1+\varepsilon \quad \mbox{for}\quad |x|\ge A.$$
By Theorem \ref{pro::2} applied to the solution
$u(x,t)-1-\varepsilon$, 
this implies that there exists a constant $C>0$ (depending on
$\varepsilon$) such that
$$u(x,t) \le 1+\varepsilon + Ct^{-\beta}.$$
Therefore, for any for $\lambda>0$ the following inequality
$$u^\lambda(x,t) \le 1+\varepsilon + Ct^{-\beta}\lambda^{-\beta}$$
holds true, which implies that 
$\displaystyle{\overline{u} = \limsup_{\lambda\to +\infty}{}^* u^\lambda}$ satisfies
$$\overline{u}(x,t) \le 1+\varepsilon \quad \mbox{for}\quad (x,t)\in
\R\times (0,+\infty).$$
Since this is true for any $\varepsilon>0$, we deduce that
$$\overline{u}(x,t) \le 1 \quad \mbox{for}\quad (x,t)\in
\R\times (0,+\infty).$$
Let us now define
$\tilde{\overline{u}} = \min\left(1,\overline{u}\right).$
By the construction,
$$\tilde{\overline{u}}(x,t)= \overline{u}(x,t) \quad \mbox{for}\quad
(x,t)\in \R\times [0,+\infty) \backslash  \left\{(0,0)\right\},$$
and, by (\ref{eq::103}), we have
$\tilde{\overline{u}}(x,0) \le H^*(x) \; \mbox{for all}\; x\in\R.$
Therefore, $\tilde{\overline{u}}$ is a subsolution of \rf{eq}--\rf{ini} on
$\R\times [0,+\infty)$ with the initial datum being the Heaviside function.

Similarly, we can show that $\displaystyle{\underline{u} = \limsup_{\lambda\to +\infty}{}_* u^\lambda}$ satisfies
$$\underline{u} \ge 0 \quad \mbox{for}\quad (x,t)\in \R\times (0,+\infty).$$
Hence, the function
$\tilde{\underline{u}} = \max\left(0,\underline{u}\right),$
which is a supersolution of \rf{eq}--\rf{ini}  on
$\R\times [0,+\infty)$ with the Heaviside initial datum. 

Finally, the conclusion of the proof is the same as in the proof of Theorem
\ref{th::110} where $\overline{u}$ (resp. $\underline{u}$) is replaced by
$\tilde{\overline{u}}$ (resp. $\tilde{\underline{u}}$). This ends the proof
of Theorem~\ref{th:m:3}. 
\end{proof}


\section{Approximate equation and gradient estimates}\label{s6}

In this section, in order to prove our gradient estimates 
of viscosity solutions stated in Theorem \ref{th:m:4},
we  replace equation (\ref{eq})  
by an approximate equation  
for which smooth solutions do exist. Indeed, with $\varepsilon>0$, 
we consider the following 
initial value problem
\begin{eqnarray}\label{eq:e}
&&u_t=\varepsilon u_{xx}  -|u_x|  \Da u \quad \mbox{on}\quad \R\times (0,+\infty),\\
\label{ini:e}
&&u(x,0)=u_0(x) \quad \mbox{for}\quad x\in \R.
\end{eqnarray}
 We have added to this equation an auxiliary viscosity term which is
stronger than $\Da u$ and $u_x$. In the case $\alpha\in (0,1]$, we will see
later (in Section \ref{s7}) 
that it is possible to pass to the limit $\varepsilon\to 0^+$ 
 in
$L^\infty(\R)$, 
which is the required convergence for the framework of viscosity
solutions. The difficulty in the case $\alpha\in (1,2)$ comes from the fact
that,  for the limit equation with $\varepsilon=0$, we are not able to
give a meaning to the product $|u_x| \ (\Da u) $ in the sense of
distributions, while it is possible when $\alpha\in (0,1]$.

Our results on qualitative properties of 
solutions to the regularized problem \rf{eq:e}--\rf{ini:e} 
are stated in the following two theorems.

\begin{theorem}[Approximate equation -- existence of  solutions]\label{th::10a}
Let $\alpha\in (0,1]$ and $\varepsilon>0$.
Given any initial datum $u_0\in C^2(\R)$ such that 
$u_{0,x}\in L^1(\R)\cap L^\infty(\R)$, there
exists a unique solution 
$u\in C(\R\times [0,+\infty))\cap C^{2,1}(\R\times (0,+\infty))$ of 
\rf{eq:e}--\rf{ini:e}.
This solution satisfies
\begin{equation}\label{ux:reg}
u_x\in C([0,T], L^{p}(\R))\cap 
C((0,T];\; W^{1,p}(\R))\cap C^1((0,T], L^{p}(\R))
\end{equation}
for every $p\in(1,\infty)$ and each $T>0$.
\end{theorem}

\begin{theorem}[Approximate equation -- decay estimates]\label{th::10b}
Under the assumptions of Theorem \ref{th::10a}, the solution $u=u(x,t)$
of \rf{eq:e}--\rf{ini:e}
satisfies 
\begin{equation}
\|u(\cdot,t)\|_\infty \le \|u_0\|_\infty, \qquad
\|u_x(\cdot,t)\|_\infty \le \|u_{0,x}\|_\infty, \label{est:first}
\end{equation}
and
\begin{equation}\label{Lp:estbis}
\|u_x(\cdot, t)\|_p\leq C_{p,\alpha} \|u_{0,x}\|_1^{\frac{p\alpha+1}{p(\alpha+1)}}
t^{-\frac{1}{\alpha+1}\left(1-\frac{1}{p}\right)},
\end{equation}
for every $p\in [1,\infty)$, all $t>0$, and  constants
$C_{p,\alpha}>0$ (see, \rf{Cpa} below), independent of $\varepsilon>0$, $t>0$ and $u_0$.
\end{theorem}

\subsection{Existence theory}

\begin{proof}[Proof of Theorem \ref{th::10a}]
Note first that 
\begin{equation} \label{DH}
\Da u=-\Lambda^{\alpha-1} \HH u_x,
\end{equation}
 where $\HH$ 
denotes the Hilbert transform defined in the Fourier variables by
$\widehat {({\HH}v)}(\xi)= i\ \mbox{sgn}(\xi)\  \widehat v(\xi).$
We recall that the Hilbert transform is bounded on the $L^{p}$-space  for
any $p\in (1,+\infty)$ (see 
\cite[Ch.~2, Th.~1]{St}), 
{i.e.}  it satisfies for any function $v\in L^{p}(\R)$ the following inequality 
\begin{equation}\label{H:est}
\|{\HH} v\|_{p} \le C_{p}  \|v\|_{p}
\end{equation}
with a constant $C_p$ independent of $v$.

For $\alpha\in (0,1)$, the operator $\Lambda^{\alpha-1}$ defined 
analogously as in \rf{Da} corresponds to the convolution with the Riesz potential 
$\Lambda^{\alpha-1}v=C_\alpha |\cdot|^{-\alpha}*v$.
Hence, by \cite[Ch.~5, Th.~1]{St}, 
 for any $p> 1/\alpha$ with $\alpha\in (0,1]$ and any function $v\in L^q(\R)$, we have 
\begin{equation}\label{R:est}
\|\Lambda^{\alpha-1}v\|_{p} \le C_{p,\alpha} \|v\|_{q} 
\quad \mbox{with}\quad
\frac{1}{q}=\frac{1}{p} +1-\alpha.
\end{equation}

Now, if $u=u(x,t)$ is a solution to \rf{eq:e}--\rf{ini:e}, 
using identity \rf{DH}, we write the initial value problem for $v=u_x$
 \begin{eqnarray}\label{eq:v}
&&v_t=\varepsilon v_{xx}  +(|v| \Lambda^{\alpha-1}\HH v)_x \quad \mbox{on}\quad \R\times (0,+\infty),\\
\label{ini:v}
&&v(\cdot,0)=v_0=u_{0,x} \in L^1(\R)\cap L^\infty(\R)
\end{eqnarray}
as well as its equivalent integral formulation
\begin{equation}\label{duh:v}
v(t)=G(\varepsilon t)*v_0+\int_0^t \partial_x G(\varepsilon(t-\tau))
*(|v| \Lambda^{\alpha-1}\HH v)\,{\rm d}\tau,
\end{equation}
with the Gauss-Weierstrass kernel $G(x,t)=(4\pi t)^{-1/2}\exp(-x^2/(4t))$.

The next step is completely standard and consists in applying the Banach 
contraction principle to equation \rf{duh:v} in a ball 
in the Banach space
$$
\mathcal{X}_T=C([0,T];\, L^1(\R)\cap L^\infty(\R)) 
$$
endowed with the usual norm $\|v\|_T= \sup_{t\in [0,T]}(\|v(t)\|_1+ \|v(t)\|_\infty)$.
Using well known estimates of the heat semigroup and inequalities 
\rf{H:est}--\rf{R:est} combined with the imbedding 
$ L^1(\R)\cap L^\infty (\R)\subset L^p(\R)$ for each $p\in [1,\infty]$, we
obtain a solution $v=v(x,t)$ 
 to equation \rf{duh:v} in the space $\mathcal{X}_T$ provided
$T>0$ is sufficiently small.  We refer the reader to, {e.g.}, \cite{AB98,BSW02} for  examples of such a reasoning. 

This solution satisfies \rf{ux:reg}
for every $p\in(1,\infty)$ and each $T>0$, by standard regularity estimates 
of solutions to parabolic equations. Moreover, following the reasoning from
\cite{AB98}, one can show that the solution is regular.

Finally, this local-in-time solution can be extended to global-in-time 
({i.e.} for all $T>0$) because of the estimates 
$\|v(t)\|_p\leq \|v_0\|_p$ for every $p\in [1,\infty]$ 
being the immediate consequence of inequalities
\rf{L1:est}, \rf{vSV}, and \rf{eq:vinfty} below.
\end{proof}
\bigskip

\subsection{Gradient estimates}
In the proof of the decay estimates of $u_x$,
we shall  require several properties of the
operator $\Da$.
First, we recall the Nash inequality for the operator $\Da$.

\begin{lemma}[Nash inequality]
Let $0<\alpha$.
There exists a constant $C_N>0$ such that
\begin{equation}\label{Nash}
\|w\|_2^{2(1+\alpha)} \leq C_N
\|\Lambda^{\alpha/2}w\|_2^2\|w\|_1^{2\alpha}
\end{equation}
for all functions $w$ satisfying  $w \in L^1(\R)$ and
$\Lambda^{\alpha/2}w\in L^2(\R)$.
\end{lemma}

The proof of inequality \rf{Nash} is given, e.g., in \cite[Lemma 2.2]{KMX}.
\medskip

Our next tool is the, so called, Stroock--Varopoulos inequality.

\begin{lemma}[Stroock--Varopoulos inequality]\label{le3.1}
Let $0\leq \alpha\leq 2$.
For every
$p>1$,  we have
\begin{equation}\label{hyper}
\int_\R (\Da w)|w|^{p-2} w \dx\geq
\frac {4(p-1)}{p^2}\int_\R\left(\Lambda^{\frac\alpha
2}|w|^{\frac p2}\right)^2\dx
\end{equation}
for all $w\in L^p(\R)$ such that $\Da w\in L^p(\R)$. If $\Da w\in
L^1(\R)$, we obtain
\begin{equation}\label{sgn est}
\int_\R (\Da w)\;{\rm sgn}\, w\,\dx\geq 0.
\end{equation}
Moreover, if $w,\, \Da w\in L^2(\R)$, it follows that 
\begin{equation}\label{+ est}
\int_\R (\Da w) w^+\,\dx\geq 0 \quad\mbox{\rm and}\quad \int_\R (\Da w)
w^-\,\dx\geq 0,
\end{equation}
where $w^+=\max\{0,w\}$  and $w^-=\max\{0,-w\}$.
\end{lemma}

Inequality \rf{hyper} is well known in the theory of sub-Markovian
operators and its statement  and the proof is given, e.g., in \cite[Th. 2.1 combined with the Beurling--Deny condition
(1.7)]{LS}.
Inequality \rf{sgn est}, called the (generalized) Kato inequality, 
is used, e.g., in \cite{DI} to construct entropy solutions of conservation laws
with a L\'evy diffusion. It can be easily deduced from \cite[Lemma 1]{DI} by an
approximation argument. 
The proof of \rf{+ est} can be found, for example, in
\cite[Prop.~1.6]{LS}.

\begin{remark}\label{Stroock-V}
Remark that inequality
(\ref{sgn est}) appears to be a limit case of (\ref{hyper}) for
$p=1$. Inequality (\ref{+ est}) for $w^+$ follows easily from (\ref{sgn
  est}), by a comparison argument, if
for instance $w\in C^\infty_c(\R)$. Finally, remark that the constant
appearing in (\ref{hyper}) is the same as for the Laplace operator 
${\partial^2}/{\partial x^2}=-\Lambda^2$.
\end{remark}

Our proof of the decay of $v(t)=u_x(t)$ is based on the following
Gagliardo--Nirenberg type inequality 

\begin{lemma}[Gagliardo--Nirenberg type inequality]
Assume that $p\in (1,\infty)$ and $\alpha>0$ are fixed and  arbitrary. For all 
$v\in L^1(\R)$ such that 
$\Lambda^{\alpha/2}|v|^{(p+1)/2}\in L^2(\R)$, the following inequality is valid
\begin{equation}\label{GN}
\|v\|_p^a\leq C_N \left\|\Lambda^{\alpha/2}|v|^{(p+1)/2} \right\|_2^2 \|v\|_1^b,
\end{equation}
where
$$
a=\frac{p(p+\alpha)}{p-1}, \quad b=\frac{p\alpha+1}{p-1},
$$
and $C_N$ is the constant from the Nash inequality \rf{Nash}.
\end{lemma}

\begin{proof}
Without loss of generality, we can assume that $\|v\|_1\neq 0$.
Substituting $w=|v|^{(p+1)/2}$ in the Nash inequality \rf{Nash} we obtain
$$
\|v\|_{p+1}^{(p+1)(1+\alpha)}\leq C_N \left\|\Lambda^{\alpha/2}|v|^{(p+1)/2} 
\right\|_2^2 
\|v\|_{(p+1)/2}^{\alpha(p+1)}.
$$
Next, it suffices to apply two particular cases of the H\"older inequality
$$
\left(
\frac{\|v\|_p}{\|v\|_1^{1/p^2}}
\right)^{p^2/(p^2-1)} \leq \|v\|_{p+1}
\quad \mbox{as well as}\quad
\|v\|_{(p+1)/2}\leq \|v\|^{p/(p+1)}_p\|v\|_1^{1/(p+1)},
$$ 
and  compute carefully all the exponents which appear on the both sides of the 
resulting inequality. 
\end{proof}


\begin{proof}[Proof of Theorem \ref{th::10b}]
The first inequality in \rf{est:first} is the immediate consequence of the comparison principle
from Theorem \ref{th::1}, because classical solutions are viscosity solutions, 
as well.
Maximum principle and an argument based on inequalities \rf{+ est}
({cf.}~\cite{KMX} for more detail) lead to the second inequality in 
\rf{est:first}. We  also discuss this inequality in Remark \ref{rem:6.8} below.

For the proof of the $L^1$-estimate
\begin{equation}\label{L1:est}
 \|u_x(t)\|_1\leq \|u_{0,x}\|_1
\end{equation}
({i.e.} \rf{Lp:estbis} with $p=1$ and $C_{p,\alpha}=1$), 
we multiply equation  
\rf{eq:v} by
${\rm sgn}\,v={\rm sgn}\,u_x$ and we integrate with respect to $x$ to obtain
$$
\frac{\rm d}{\dt}\int_\R|v|\dx = \varepsilon \int_\R v_{xx}{\rm sgn}\,v\dx+
\int_\R \left((\Lambda^{\alpha-1}{\HH}v) |v| 
\right)_x {\rm sgn}\, v\dx.
$$
The first term on the right hand side is nonpositive by the Kato inequality 
({i.e.} \rf{sgn est} with $\alpha=2$) hence we skip it.
Remark that (formally)
\begin{equation*}
\begin{split}
\int_{\R}  \left((\Lambda^{\alpha-1}{\HH}v) |v|  \right)_x {\rm
  sgn}\, v \dx
&= {\int_{\R}  (\Lambda^{\alpha-1}{\HH}v) v_x ({\rm  sgn\,} v)^2 +
(\Lambda^{\alpha-1}{\HH}v_x) v} \dx\\
& =  \int_{\R}   \left((\Lambda^{\alpha-1}{\HH}v) v\right)_x \dx = 0.
\end{split}
\end{equation*}
Now, approximating the sign function in a standard way by  ${\rm 
sgn_\delta}(z)=z/\sqrt{z^2+\delta}$, integrating by parts,
and passing to the limit $\delta \to 0^+$, one can show rigorously
that the second term on right hand 
side of the above inequality is nonpositive.  
This completes the proof of \rf{L1:est} with $p=1$. 
\medskip

Next, we multiply equation in \rf{eq:v} by $|v|^{p-2}v$ with $p>1$ to get 
$$
\frac{1}{p}\frac{\rm d}{\dt}\int_\R|v|^p\dx 
=\varepsilon \int_\R v_{xx}|v|^{p-2}v\dx+
\int_\R \left((\Lambda^{\alpha-1}{\HH}v) |v| \right)_x |v|^{p-2} v\dx.
$$
We drop the first term on the right hand side, because it is 
nonpositive by \rf{hyper} with $\alpha=2$.
Integrating by parts and using the elementary identity
$$
|v|\left(|v|^{p-2}v\right)_x=\frac{p-1}{p}\left(|v|^{p-1}v\right)_x,
$$
we transform the second quantity on the right hand side as follows
$$
\int_\R \left((\Lambda^{\alpha-1}{\HH}v) |v| \right)_x |v|^{p-2} v\dx
=-\frac{p-1}{p}\int_\R (\Da v)|v|^{p-1}v\dx.
$$
Consequently, by the Stroock--Varopoulos inequality \rf{hyper} (with the exponent  $p$ replaced by $p+1$),  we obtain 
\begin{equation}\label{vSV}
\frac{\rm d}{\dt} \|v(t)\|_p^p\leq -\frac{4p(p-
1)}{(p+1)^2}\left\|\Lambda^{\alpha/2}|v|^{(p+1)/2}\right\|_2^2.
\end{equation}
Hence, the interpolation inequality \rf{GN} combined with \rf{L1:est} lead to 
the following inequality for $\|v(t)\|_p^p$ 
\begin{equation}\label{diff:in}
\frac{\rm d}{\dt} \|v(t)\|_p^p\leq -\frac{4p(p-1)}{(p+1)^2} 
\left(C_N\|v_0\|_1^{(p\alpha+1)/(p-1)}  \right)^{-1}
\left(\|v(t)\|_p^p\right)^{(p+\alpha)/(p-1)}.
\end{equation}

Recall now that if a nonnegative (sufficiently smooth function) $f=f(t)$ 
satisfies, for all $t>0$,
the inequality 
$f'(t)\leq -Kf(t)^\beta$ with constants $K>0$ and $\beta>1$, then $f(t)\leq 
(K(\beta-1)t)^{-1/(\beta-1)}$.
Applying this simple result to the differential inequality \rf{diff:in}, we 
complete the proof of the $L^p$-decay
estimate \rf{Lp:estbis} with the constant 
\begin{equation}\label{Cpa}
C_{p,\alpha}=\left(C_N^{-1}\frac{4p(\alpha+1)}{(p+1)^2} \right)^{-
\frac{1}{\alpha+1}
\left(1-\frac{1}{p}\right)},
\end{equation}
where $C_N$ is the constant from the Nash inequality \rf{Nash}.
\end{proof}

\begin{remark}\label{rem:6.8}
Note that, for every fixed $\alpha$, we have 
$\lim_{p\to\infty}C_{p,\alpha}=+\infty$.
By this reason, we are not allowed to pass directly to the limit $p\to+\infty$ in 
inequalities \rf{Lp:estbis}
(as  was done in, e.g.,  \cite[Th. 2.3]{KMX}) in order to obtain a decay estimate of 
$v(t)$ in the $L^\infty$-norm. Nevertheless, 
using  \rf{diff:in} we immediately deduce the inequality 
$\|v(\cdot,t)\|_p \leq \|v_0(\cdot)\|_p$ valid for every $p\in (1, \infty)$.
Hence,  passing to the limit   $p\to +\infty$ we
get
\begin{equation}\label{eq:vinfty}
\|v(\cdot,t)\|_{\infty} \leq \|v_0(\cdot)\|_\infty.   
\end{equation}

In general, we cannot hope to get 
a decay estimate of $\|v(\cdot,t)\|_{\infty}$  better  than that in 
(\ref{eq:vinfty}), because each constant  is a solution of  equation \rf{eq}. 
Moreover, one can show  that if $v$ is constant on
an interval at the initial time, then it will stay equal to the same constant on an interval 
depending on $t$, because of the finite propagation phenomenon that can be seen for the self-similar profile.
\end{remark}

\section{Passage to the limit  and proof of Theorem \ref{th:m:4}}\label{s7}

Now, we are in a position to complete the proof of the gradient estimates
\rf{eq::decay}. First, we show that form the sequence 
$\{u^\varepsilon\}_{\varepsilon>0}$ of solutions to the approximate problem
\rf{eq:e}--\rf{ini:e}  one can extract, {\it via} 
the Ascoli--Arzel\`a theorem, a subsequence converging uniformly.
Theorem \ref{th:stab} on the stability and Remark \ref{remark:plus}
imply that the limit function is a viscosity
solution to \rf{eq}--\rf{ini}. Passing to the limit $\varepsilon\to 0^+$ in
inequalities \rf{est:first} and \rf{Lp:estbis} we complete our reasoning.

\begin{proof}[Proof of Theorem \ref{th:m:4}]
First, let us suppose  that $u_0\in C^\infty(\R)\cap W^{2,\infty}(\R)$ 
with $u_{0,x}\in L^1(\R)\cap L^\infty(\R)$. Denote by 
$u^\varepsilon=u^\varepsilon(x,t)$ the corresponding solution to the 
approximate problem with $\varepsilon>0$. 
\medskip

 {\it Step 1: Modulus of continuity in space.}
Under this additional assumption, we have 
\begin{equation}\label{eq::200b}
\|u^\varepsilon_x(\cdot, t)\|_p\leq C_{p} t^{-\gamma_p}   
\end{equation}
with $C_{p}= C_{p,\alpha}
\|u_{0,x}\|_1^{\frac{p\alpha+1}{p(\alpha+1)} }$ and $\gamma_p=\frac{1}{\alpha+1}\left(1-\frac{1}{p}\right)$.
The Sobolev imbedding theorem implies that there exist 
some $\beta\in (0,1)$ and $C_0>0$ such that 
\begin{equation}\label{eq::holder}
|u^\varepsilon(x+h,t)-u^\varepsilon(x,t)|\le |h|^\beta C_0C_{p} t^{-\gamma_p}.   \end{equation}
\medskip

 {\it Step 2: Modulus of continuity in time.}
Let us consider a nonnegative function $\varphi\in C^\infty(\R)$ 
with $\mbox{supp}\,\varphi\subset [-1,1]$ such that  
$\int_{\R}\varphi(x)\dx=1$, and for  any $\delta>0$ set
$\varphi_\delta(x)=\delta^{-1}\varphi(\delta^{-1}x)$. Then, multiplying
(\ref{eq:e}) by $\varphi_\delta$ and integrating in space, we get
\begin{equation*}
\begin{split}
\frac{\rm d}{\dt}\left(\int_{\R} u^\varepsilon(x,t)\varphi_\delta(x)\dx\right)=&\ \varepsilon
\la u^\varepsilon(\cdot,t),(\varphi_\delta)_{xx}\ra \\
&- \int_{\R} \varphi_\delta(x)\ |u^\varepsilon_x(x,t)| 
({\mathcal H}\ \Lambda^{\alpha-1}\ u^\varepsilon_x(x,t))\dx,
\end{split}
\end{equation*}
and then with ${1}/{p}+{1}/{p'}=1$
\begin{equation}
\begin{split}
\left|\frac{\rm d}{\dt}\left(\int_{\R}
    u^\varepsilon(x,t)\varphi_\delta(x)\dx\right)\right|
&\le \varepsilon \|u^\varepsilon(\cdot,t)\|_{\infty}\|(\varphi_\delta)_{xx}\|_{1} \\
&+
    \|\varphi_\delta\|_\infty 
\|u^\varepsilon_x(\cdot,t)\|_p\| \|{\mathcal H}\ \Lambda^{\alpha-1}\
    u^\varepsilon_x(\cdot,t)\|_{p'}. 
\end{split}
\end{equation}
Here, we have used relation \rf{DH}. 
Combining inequalities \rf{H:est} and \rf{R:est}  with estimate (\ref{eq::200b}), we get for $p'> 1/\alpha$
$$\|{\mathcal H}\ \Lambda^{\alpha-1}\
    u^\varepsilon_x(\cdot,t)\|_{p'} \le C_{p'}C_{p',\alpha}C_{q} t^{-\gamma_q}.$$
Then for any bounded time interval $I\subset (0,+\infty)$ there exists a
    constant $C_{I,\delta}$ such that for all $t\in I$, we have for any
    $\varepsilon\in (0,1]$
$$\left|\frac{\rm d}{\dt}\left(\int_{\R}
    u^\varepsilon(x,t)\varphi_\delta(x)\dx\right)\right|
\le C_{I,\delta}.$$
Now, for any $t,t+s\in I$, we get
$$\left|\int_{\R}
    u^\varepsilon(x,t+s)\varphi_\delta(x)\dx - \int_{\R}
    u^\varepsilon(x,t)\varphi_\delta(x)\dx \right| \le |s| C_{I,\delta}.$$
Therefore, the following estimate
\begin{equation*}
\begin{split}
&|u^\varepsilon(0,t+s)-u^\varepsilon(0,t)|\\
&
\le |s| C_{I,\delta} 
+ \int_{\R}\varphi_\delta(x)\dx \cdot 
 \sup_{x\in[-\delta,\delta]} \big(
  |u^\varepsilon(x,t+s)-u^\varepsilon(0,t+s)| 
+ |u^\varepsilon(x,t)-u^\varepsilon(0,t)|  \big)
\end{split}
\end{equation*}
holds true. 
Using the H\"{o}lder estimate (\ref{eq::holder}), we deduce that there
exists a constant $C_I$ depending on $I$, but independent of $\delta$ and
of $\varepsilon \in (0,1]$, such that
$$|u^\varepsilon(0,t+s)-u^\varepsilon(0,t)|\le |s| C_{I,\delta} + C_I
\delta^\beta.$$
Since the above inequality is true for any $\delta$, this shows the existence
of a modulus of continuity $\omega_I$ satisfying
$$|u^\varepsilon(0,t+s)-u^\varepsilon(0,t)|\le \omega_I(|s|) \quad
\mbox{for any}\quad t,t+s\in I.$$
By the translation invariance of the problem, this estimate is indeed true
for any $x\in \R$, i.e.
\begin{equation}\label{eq::holdert}
|u^\varepsilon(x,t+s)-u^\varepsilon(x,t)|\le \omega_I(|s|) \quad
\mbox{for any}\quad t,t+s\in I,\quad x\in\R.   
\end{equation}

 {\it Step 3: Convergence as $\varepsilon\to 0^+$.}
From estimates (\ref{eq::holder}) and (\ref{eq::holdert}), and using the 
Ascoli--Arzel\`a theorem and  the Cantor diagonal argument, 
we deduce that there exists a subsequence (still
denoted  $\{u^\varepsilon\}_\varepsilon$) which converges to a limit $u\in
C(\R\times (0,+\infty))$.  By the stability result in Theorem \ref{th:stab}
(see also Remark \ref{remark:plus}), 
we have that $u$ is a viscosity solution of (\ref{eq}) on
$\R\times (0,+\infty)$.
\medskip

 {\it Step 4: Checking the initial conditions for $u_0$ smooth.}
Remark that for $u_0\in W^{2,\infty}$ we can use the barriers given in
(\ref{eq::200}) with some constant $C>0$ uniform in $\varepsilon\in (0,1]$.
This ensures that $u$ is continuous up to $t=0$ and
satisfies
$u(\cdot,0)=u_0,$
so this proves the result under additional assumptions.
\medskip

 {\it Step 5: General case.}
The proof in the case of less regular initial conditions
simply follows by an approximation argument as was in the
proof of Theorem \ref{th:exis}. 
\medskip

{\it Step 6: Gradient estimates.} To pass to the limit 
$\varepsilon\to 0^+$ in estimates \rf{Lp:estbis}, we use 
the inequality
\begin{equation}\label{eq:h}
h^{-1}\|u^\varepsilon(\cdot+h,t)-u^\varepsilon(\cdot,t)\|_p
\leq \|u^\varepsilon_x(\cdot,t)\|_p 
\end{equation}
with fixed $h>0$. Hence, by the Fatou lemma combined with the pointwise
convergence of $u^\varepsilon$ toward $u$, we deduce from \rf{eq:h} and
\rf{Lp:estbis} that
 \begin{equation*}\label{eq:h2}
h^{-1}\|u(\cdot+h,t)-u(\cdot,t)\|_p 
\leq C_{p,\alpha} \|u_{0,x}\|_1^{\frac{p\alpha+1}{p(\alpha+1)}}
t^{-\frac{1}{\alpha+1}\left(1-\frac{1}{p}\right)}
\end{equation*}
for all $h>0$. For every fixed $t>0$, the sequence 
$\{h^{-1}(u(\cdot+h,t)-u(\cdot,t))\}_{h>0}$ is bounded
in $L^p(\R)$ and converges (up to a subsequence) weakly in $L^p(\R)$
toward $u_x(\cdot,t)$ (see, {e.g.}, \cite[Ch.~V, Prop.~3]{St}). 
Using the well-known property of a weak convergence
in Banach spaces we conclude
$$
\|u_x(\cdot,t)\|_p\leq \liminf_{h\to 0^+}
h^{-1}\|u(\cdot+h,t)-u(\cdot,t)\|_p 
\leq C_{p,\alpha} \|u_{0,x}\|_1^{\frac{p\alpha+1}{p(\alpha+1)}}
t^{-\frac{1}{\alpha+1}\left(1-\frac{1}{p}\right)}.
$$
This finishes the proof of Theorem~\ref{th:m:4}. 
\end{proof}

\bigskip
\bigskip


\bibliographystyle{amsplain}

\end{document}